\documentclass{article}

\usepackage{amsmath}
\usepackage{amssymb}
\usepackage{amstext}
\usepackage{amsrefs}
\usepackage{stmaryrd}
\usepackage{pdflscape}

\usepackage{mathtools}
\usepackage{tensor}
\usepackage{eufrak}
\usepackage[colorlinks]{hyperref}
\usepackage{amsthm}
\usepackage[all,2cell]{xy}
\UseTwocells
\usepackage{bbm}
\usepackage{bbold}
\usepackage{lscape}
\usepackage[titletoc,title]{appendix}

\allowdisplaybreaks[1]

\newcommand{\C}{\mathcal C}

\newcommand{\D}{\mathcal D}

\newcommand{\ot}{{\otimes}}
\newcommand{\pl}{{\oplus}}

\newcommand{\op}{\mathsf{op}}
\newcommand{\cop}{\mathsf{cop}}

\newcommand{\e}{\mathsf{e}}
\newcommand{\db}{\mathsf{d}}

\newcommand{\ka}{\kappa}
\newcommand{\la}{\lambda}
\newcommand{\p}{{\cdot}}

\newcommand{\nrr}{\nu^r_R}
\newcommand{\nlr}{\nu^l_R}
\newcommand{\nrl}{\nu^r_L}
\newcommand{\nll}{\nu^l_L}
\newcommand{\I}{\mathbbm 1}

\newtheorem{theorem}{Theorem}
\newtheorem{definition}[theorem]{Definition}
\newtheorem{proposition}[theorem]{Proposition}
\newtheorem{lemma}[theorem]{Lemma}
\newtheorem{remark}[theorem]{Remark}

\newtheorem{example}[theorem]{Example}

\begin{document}

\title{A note on Frobenius monoidal functors on autonomous categories}

\author{Adriana Balan
\thanks{This work was supported by a grant of the Romanian National Authority for Scientific Research, CNCS--UEFISCDI, project number PN-II-RU-TE--2012--3--0168.} 
\thanks{Simion Stoilow Institute of Mathematics of the Romanian Academy, Research group of the project TE--3--0168, P.O. Box 1-764, RO-014700, Bucharest, Romania.}
\thanks{University Politehnica of Bucharest, Romania, adriana.balan@mathem.pub.ro}
}
\maketitle

\begin{abstract}

Frobenius monoidal functors preserve duals. We show that conversely, (co)monoidal functors between autonomous categories which preserve duals are Frobenius monoidal. We apply this result to linearly distributive functors between autonomous categories. 

\end{abstract}


\section{Introduction}

There is an old result, going back to~\cite{js:braided}, 
saying that strong monoidal functors preserve duals. However, not all (co)monoidal functors preserving duals are necessarily strong; 
examples being the Frobenius monoidal functors introduced in~\cite{day-pastro:Frob}. These are simultaneously monoidal and comonoidal functors, subject to certain coherence conditions. 

The notion of an autonomous functor between left/right autonomous categories, introduced in~\cite{dms-da}, formalizes the property of preserving duals. Explicitly, a (co)monoidal functor $F$ between (left) autonomous categories is autonomous if there is a natural isomorphism $FS \cong SF$ satisfying two coherence conditions, where $S$ denotes the contravariant functor taking (left) duals. Frobenius monoidal functors between autonomous categories are autonomous~\cite{day-pastro:Frob}. In this paper we show that conversely, (co)monoidal functors between autonomous categories carry also a monoidal, respectively comonoidal structure making them Frobenius monoidal. 

\bigskip

Frobenius monoidal functors arise as a (degenerate) specialization of linear functors between linearly distributive categories, to the case where both domain and codomain categories have equal tensor products -- that is, are monoidal categories.\footnote{Note however that in~\cite{egger:star} the notion of a Frobenius monoidal functor is considered between genuine linearly distributive categories, while in~\cite{blute:deep} the same appears under the name of degenerate linear functor.} More in detail, a Frobenius monoidal functor produces a linear functor with \emph{equal} components. 

It appears natural to ask to what extent a linear functor between monoidal categories is induced by a Frobenius monoidal functor as above. However, imposing an equality rather than an isomorphism is usually considered \emph{evil} in category theory, thus the question above should be rephrased as which are the necessary coherence conditions ensuring that the components of a linearly distributive functor are isomorphic if and only if one of them (or both) is  Frobenius monoidal.

Cockett and Seely's characterization of linear functors in presence of dualities (also called negations)~\cite{cockett-seely:lf}, enhanced by the equivalence Frobenius monoidal functor -- autonomous functor, will allow us to show that for a linear functor between autonomous categories, one of the components being Frobenius monoidal/autonomous forces also the other to be so, and that this is equivalent to the existence of a monoidal-comonoidal isomorphism between them, or an isomorphism compatible with the linear structure. 

However, for general monoidal categories, the situation can be completely arbitrary: we provide a simple example of a linear functor having only one component Frobenius monoidal, and another example with both components Frobenius monoidal, but non-isomorphic (Example~\ref{ex:lin}). 

\bigskip

The paper is organized as follows: Section~\ref{sect:prelim} is dedicated to a review of the notions of monoidal categories, monoidal functors and linear functors, autonomous categories. Subsequently, the main Subsection~\ref{subsect:aut funct} focuses on autonomous and Frobenius monoidal functors and culminates with their equivalence in Theorem~\ref{thm:aut->frob}; see also Theorem~\ref{cor:frob} for several equivalent characterizations of such functors. The last subsection provides an application to linear functors between autonomous categories. 

Appendix \ref{app:prelim} expands Section~\ref{sect:prelim} with more details, to be used in Appendix~\ref{app:proofs} for the diagrammatic proofs. These being quite large, were not included in the main body of the paper. 


\section{Monoidal and linear functors}\label{sect:prelim}

In this section we briefly recall the main notions needed in the sequel. More details can be found in Appendix~\ref{app:prelim}.


\subsection{Monoidal categories and functors}

All monoidal categories will have tensor product denoted $\ot$ and unit $\I$.  If $\C$ is a monoidal category, the reversed tensor product $X\ot^{\mathsf{rev}} Y=Y\ot X$ determines another monoidal structure on $\C$, that we shall denote by $\C^\cop$. 
The opposite category also becomes monoidal, with either the original monoidal product $\ot$, in which case we refer to it as $\C^\op$, or with the reversed monoidal product $\ot^{\mathsf{rev}}$. We shall then use the notation $\C^{\op, \cop}$. All the above mentioned monoidal categories have the same unit object ${\I}$. For more details on monoidal categories, we refer to~\cite{maclane}. 

\medskip

We shall in the sequel omit the associativity and unit constraints, writing as the monoidal categories would be strict \cite{maclane}. The identity morphism will be always denoted by $1$, the carrier being obvious from the context. Also, we shall not use labels on (di)natural transformations to avoid notational overcharge.

\medskip

Because many of the proofs rely on commutative diagrams, to increase their readability, we shall label diagrams by $(N)$ if these commute by naturality, and by $(M)$ if the commutativity is due to monoidal functoriality, as in $(1\ot g) (f\ot 1) = (f\ot 1)(1\ot g)$. Otherwise, we shall refer to previously labeled relations. 
 
\bigskip
 
For a monoidal functor $F:\C\to \D$ between monoidal categories, we shall employ small letters to denote the structural morphisms $f_2:FX \ot FY \to F(X\ot Y)$ and $f_0:{\I}\to F{\I}$, while in case $F$ is comonoidal we shall write $F_2:F(X\ot Y) \to FX \ot FY$ and $F_0:F{\I} \to {\I}$ (with capital letters to emphasize the difference). 

\medskip

There are well-known notions of monoidal and comonoidal natural transformations. Less encountered in the literature, but needed in this paper,  are natural transformations between functors with different monoidal orientation.  First, recall from~\cite[Section 2.3]{grandis-pare} the double category of monoidal categories, having as horizontal arrows the monoidal functors and as vertical arrows the comonoidal ones. A square
\[
\xymatrix{\cdot \drtwocell<\omit>{\alpha} \ar[r]^{F} \ar[d]_{G} & \cdot \ar[d]^{K} \\ \cdot \ar[r]_{H} & \cdot } 
\]
is then given by a natural transformation $\alpha: K F  \to HG$, subject to the coherence conditions 
\begin{equation*}
\vcenter{
\xymatrix@C=10pt@R=15pt{ & K(FX \ot FY) \ar[dl]_{Kf_2} \ar[dr]^{K_2}  & & & K{\I} \ar[dl]_{Kf_0} \ar[dr]^{K_0} &  \\
KF(X \ot Y) \ar[d]_{\alpha} && KFX \ot KFY \ar[d]^{\alpha \ot \alpha} & KF{\I} \ar[d]_{\alpha} && {\I} \ar@{=}[d] \\
HG(X \ot Y) \ar[dr]_{HG_2} && HGX \ot HGY \ar[dl]^{h_2} & HG{\I} \ar[dr]_{HG_0} && {\I} \ar[dl]^{h_0} \\
& H(GX \ot GY) & & & H{\I} }
}
\end{equation*}  

In case both $K$ and $H$ are the identity functors, we call such $\alpha:F\to G$ a \emph{monoidal-comonoidal natural transformation}.\footnote{Some authors call such natural transformations lax/colax (monoidal).} For the sake of completeness, we spell out explicitly that $\alpha:F \to G$ should satisfy \begin{align}
\label{lax-colax} & \vcenter{
\xymatrix{FX \ot FY \ar[d]_{f_2} \ar[r]^{\alpha \ot \alpha} & GX \ot GY \\
F(X \ot Y) \ar[r]^\alpha & G(X \ot Y) \ar[u]_{G_2}
}} & \vcenter{
\xymatrix{
{\I} \ar@{=}[r] \ar[d]_{f_0} & {\I} \\
F{\I} \ar[r]^{\alpha} & G{\I} \ar[u]_{G_0}
}}
\end{align}
By taking instead $F$ and $G$ to be identity, the resulting natural transformation $\alpha:K \to H$ will be called \emph{comonoidal-monoidal}. 

\medskip


Finally, we recall from~\cite{day-pastro:Frob} 
 that a functor $F:\C \to \D$ between monoidal categories is called \emph{Frobenius monoidal} if:  
\begin{itemize}
\item It is a monoidal functor, with 
\[ f_{2}:FX\ot FY \to F(X \ot Y)\ , \ f_0:{\I}\to F{\I} 
\]
\item It is a comonoidal functor, with 
\[
F_{2}:F(X\ot Y) \to FX \ot FY\ , \ F_0:F{\I}\to {\I}
\]
\end{itemize}
and satisfies the compatibility conditions expressed in the diagrams below: 

\begin{align}
\vcenter{\xymatrix{
FX \ot F(Y \ot Z) \ar[d]_{f_2} \ar[r]^{1\ot F_2} & FX \ot FY \ot FZ \ar[d]^{f_2\ot 1}\\
F(X\ot Y\ot Z) \ar[r]^{F_2} & F(X\ot Y)\ot FZ}}\label{eq2:Frob} \\
\vcenter{\xymatrix{
F(X \ot Y) \ot FZ \ar[d]_{f_2} \ar[r]^{F_2\ot 1} & FX \ot FY \ot FZ \ar[d]^{1\ot f_2} \\F(X\ot Y\ot Z) \ar[r]^{F_2} & FX\ot F(Y\ot Z) 
} } \label{eq1:Frob}
\end{align}

The simplest example of a Frobenius monoidal functor is a strong monoidal one, in which case the structural monoidal/comonoidal morphisms are inverses to each other. 

As for the properties of Frobenius monoidal functors, we mention that such functors preserve duals~\cite{day-pastro:Frob}. To this point we shall come back later (Remark~\ref{FMF pres dual}) and we shall see that the converse also holds, in the sense that (co)monoidal functors on \emph{autonomous} categories which coherently preserve duals (call them autonomous functors) are Frobenius monoidal (Theorems~\ref{thm:aut->frob} and~\ref{cor:frob}).


\subsection{Linear functors between monoidal categories} \label{sect:lin funct mon cat}

We shall in the sequel encounter pairs of functors between monoidal categories, one of them being monoidal and the other comonoidal, subject to several coherence conditions. Such a pair has been called a linearly distributive functor and it makes sense in a more general context than monoidal categories, namely linearly distributive categories.\footnote{These have been introduced by Cockett and Seely in \cite{cockett-seely:wdc:short} and \cite{cockett-seely:wdc} as to provide a categorical settings for linear logic.}~In short, a linearly distributive category is a category $\C$ equipped with two monoidal structures $(\C,\ot,{\I})$ and $(\C,\pl,\mathbb 0)$, 
and two natural transformations $A\ot (B\pl C)\to (A\ot B)\pl C$, $ 
(A\pl B)\ot C \to A\pl(B\ot C)$, subject to several naturality coherence conditions that make the monoidal structures work together \cite{cockett-seely:wdc}. Any monoidal category $(\C, \ot, {\I})$ is a (degenerate) linearly distributive category, with $\ot=\pl$ and ${\I}=\mathbb 0$. This is our case of interest, that we shall pursue from now on in the sequel. 

Linearly distributive functors (in short, linear functors) between linearly distributive categories were defined in \cite{cockett-seely:lf}.\footnote{To not be confounded with another notion of linear functor, namely an (enriched) functor between categories enriched over vector spaces.} However, for our purposes, it will be enough to only consider linear functors between degenerate linearly distributive categories (monoidal categories). Thus, a \emph{linear functor} consists of a pair of functors $R, L:\C\to \D$ between monoidal categories $\C$ and $\D$, such that $R$ is monoidal and $L$ is comonoidal, with structure maps 
\begin{align*} 
r_{2}:R X \ot R Y \to R (X\ot Y), \qquad & r_0:{\I}\to R {\I}\\
L_{2}:L(X \ot  Y) \to L X \ot L Y, \qquad & L_0:L {\I}\to {\I}
\end{align*} 
such that there are four natural transformations, called (co)strengths
\begin{align*} 
&\nu_R^r:R (X\ot Y)\to L X\ot R Y,  \qquad \nu_R^l:R (X\ot Y)\to R X \ot L Y \\
&\nu_L^r:R X \ot L Y \to L (X\ot Y), \qquad \nu_L^l: L X \ot R Y \to L (X\ot Y) 
\end{align*} 
expressing how $R$ and $L$ (co)act on each other, subject to the several coherence conditions \cite{cockett-seely:lf} that can be found in~\ref{linear}.


\begin{example}

From any strong monoidal functor $(U,u_2,u_0):\C\to \D$ between monoidal categories, one can obtain a linear functor by setting $R=L=U$, with (co)strengths given by $\nu_R^l=\nu_R^r=u_2^{-1}$, $\nu_L^l=\nu_L^r=u_2$. More generally, any Frobenius monoidal functor $(F, f_2, f_0, F_2, F_0):\C \to \D$ provides a linear functor $(R, L)$ with equal components $R=L=F$, such that $\nrr=\nlr=L_2=F_2$ and $\nlr=\nll=r_2=f_2$ \cite{egger}. In fact, one can easily see that the converse also holds: any linear functor $(R, L)$ with equal components $R=L$, such that $\nrr=\nlr=L_2$ and $\nlr=\nll=r_2$, induces a Frobenius monoidal functor. 

\end{example}

In light of the above example, it appears natural to ask whether a converse of the above holds, in the following sense: for $(R,L)$ is an arbitrary linear functor between monoidal categories, if one of the components, say $R$, is Frobenius monoidal, does it follow that the other component $L$ is also Frobenius monoidal? and does an isomorphism between them exist (maybe subject to several coherence conditions)? 

This is not true in general, and we provide below an example. However, in case both categories involved are autonomous, we shall see in the last part of the paper that such a result does hold.

\begin{example}\label{ex:lin}

Consider the posetal category $(\mathbb N, \leq)$ of natural numbers with the usual order. This is a (strict) monoidal category under addition, with zero as unit. 

A monoidal functor $R:\mathbb N \to \mathbb N$ is a monotone and subadditive function, in the sense that $0\leq R(0)$ (obviously true!) and $R(m)+R(n) \leq R(m+n)$, for all $m,n\in \mathbb N$. A strong monoidal functor is a monotone morphism of monoids. Notice that a Frobenius monoidal functor between partially ordered monoids is the same as a strong monoidal functor. For our example, consider $R$ to be the constant mapping to zero.

A comonoidal functor $L:\mathbb N \to \mathbb N$ is again a monotone function, but satisfying the reversed inequalities $L(0)\leq 0$ (of course, this implies $L(0)=0$) and $L(m+n)\leq L(m)+L(n)$ for all $m,n\in \mathbb N$. Take $L$ to be the modified successor function, $L(0)=0$ and $L(n)=n+1$ for $n\neq 0$. 

The inequalities $R(m+n)\leq R(m)+L(n)\leq L(m+n)$ and $R(m+n)\leq L(m)+R(n) \leq L(m+n)$, which hold for all $m,n\in \mathbb N$, play the role of (co)strengths and ensure that the chosen pair $(R, L)$ is a linear functor on the (posetal) monoidal category of natural numbers. Notice that $R$ is Frobenius monoidal, while $L$ is only comonoidal. 

Notice that if we take instead $L$ to be the identity function, we still obtain a linear functor $(R,L)$, this time with the (non-isomorphic!) components being both Frobenius monoidal. 

\end{example}


\subsection{Autonomous categories}\label{sec:autonomous}

We quickly review below the basics on autonomous categories; more details can be found in the~\ref{app:aut} and in the references~\cite{kelly:many,kelly-laplaza:coherence,js:braided,freyd-yetter:coherence}.

\bigskip

A \emph{left dual} of an object $X$ in a monoidal category $\C$ consists of another object $SX$,\footnote{Left duals are also denoted ${}^*X$, ${}^\perp X$, or ${}^\vee X$. Our notation is borrowed from \cite{pastro-street}, where $S$ is reminiscent of the (left) star operator on a $*$-autonomous category. Also, we wanted to employ the same type of notation as for functors precisely to emphasize the (contravariant) functorial nature of the process of taking (left) duals. }
together with a pair of arrows $\db:{\I}\to X\ot SX$, $\e:SX\ot X\to {\I}$, satisfying the relations 
\begin{equation}
\label{left-dual}
\vcenter{\xymatrix@C=35pt{X\ar[r]^-{\db \ot 1} \ar@{=}[dr] & X\ot SX\ot X \ar[d]^{1\ot \e} \\ & X}}
\qquad 
\vcenter{\xymatrix@C=35pt{SX\ar[r]^-{1\ot \db} \ar@{=}[dr] & SX\ot X\ot SX \ar[d]^{\e\ot 1} \\ & X}}
\end{equation}
A monoidal category is called \emph{left autonomous} if each object has a left dual.

\begin{example}\label{ex:lin left dual}
Let $(R, L)$ be a linear functor between monoidal categories. Then $L\I$ is a left dual for $R\I$, with morphisms $\xymatrix@1{\I \ar[r]^{r_0} & R\I \ar[r]^-{\nlr} & R\I \ot L\I}$ and $\xymatrix@1{L\I \ot R\I \ar[r]^-{\nll} & L\I \ar[r]^{L_0} & \I}$.
\end{example}

Each object $X\in \C$ with left dual induces adjunctions $(-) \ot X \dashv (-)\ot SX :\C \to \C$ and $SX \ot (-) \dashv X \ot (-):\C \to \C$. Consequently, a left dual of an object, if it exists, is unique up to isomorphism.

Let $f:X\to Y$ an arrow between objects with left duals. Then the composite 
\begin{align}
\label{transpose}
\xymatrix{SY \ar[r]^-{1\ot \db} &SY \ot X\ot SX\ar[r]^{1\ot f\ot 1} &  SY\ot Y \ot SX\ar[r]^-{\e\ot 1} & SX}
\end{align}
is called the (left) dual arrow of $f$ (or the (left) transpose of $f$). We shall denote it by $Sf: SY \to SX$. Assuming a choice of duals in a left autonomous category $\C$, the assignments $X \mapsto SX, f\mapsto Sf$ extend functorially to a strong monoidal functor $S:\C\to \C^{\op, \cop}$. 

\bigskip

A \emph{right dual} of an object $X$ in a monoidal category $\C$ is an object $S'X$, together with arrows $\db':{\I}\to S'X\ot X$, $\e':X\ot S'X\to {\I}$, satisfying the relations 
\begin{equation}\label{eq:rightdual}
\vcenter{\xymatrix@C=35pt{X\ar[r]^-{1\ot \db'_X} \ar@{=}[dr] & X\ot S'X\ot X \ar[d]^{\e'_X\ot 1} \\ & X}}
\qquad 
\vcenter{\xymatrix@C=35pt{S'X\ar[r]^-{\db'_X\ot 1} \ar@{=}[dr] & S'X\ot X\ot S'X \ar[d]^{1\ot \e'_X} \\ & X}}
\end{equation}
A monoidal category is called \emph{right autonomous} if each object has a right dual.

\begin{example}\label{ex:lin right dual}
Let $(R, L)$ be a linear functor between monoidal categories. Then $L\I$ is not only a left dual for $R\I$, as seen earlier, but also a right dual for $R\I$, with morphisms $\xymatrix@1{\I \ar[r]^{r_0} & R\I \ar[r]^-{\nrr} & L\I \ot R\I }$ and $ \xymatrix@1{R\I \ot L\I \ar[r]^-{\nrl} & L\I \ar[r]^{L_0} & \I}$. 
\end{example}

All the results stated above for left duals apply to right duals. In particular, once a choice of right duals is assumed, $S'$ becomes a strong monoidal functor $S':\C^{\op,\cop}\to \C$.

\medskip

A category is called \emph{autonomous} if it is both left and right autonomous. 
 Then $S$ and $S'$ form a contravariant pair of adjoint equivalences $S\dashv S':\C^{\op,\cop} \to \C$. We shall denote the (monoidal) unit and counit by $\alpha:X \to S'SX$, respectively $\beta:X \to SS'X$.

\section{Frobenius monoidal and linear functors on autonomous categories} \label{sec:main}

This section contains the results of this paper. In the first subsection, we show that there is a one-to-one correspondence between autonomous and Frobenius monoidal functors on autonomous categories. In the second one, we apply the previously results to see that for a linear functor between autonomous categories, one of its components is Frobenius monoidal if and only if the other one is, if and only if there is a natural isomorphism between them compatible with the four (co)strengths. 


\subsection{Autonomous and Frobenius monoidal functors}\label{subsect:aut funct}

We begin by giving the precise definition of what it means for a monoidal functor to preserve duals, that is, to be an autonomous functor. This notion has been introduced in~\cite{dms-da}, in the more general context of autonomous pseudomonoids in (autonomous) monoidal bicategories.

\begin{definition}
Let $(F,f_2,f_0):\C \to \D$ a monoidal functor between left autonomous categories. We say that $F$ is \emph{left autonomous}\footnote{We shall see in the sequel that in case $\C$ and $\D$ are both left and right autonomous, $F$ being left autonomous is the same as being right autonomous, so it will be unambiguous to drop off the adjective "left/right" and simply call such a functor \emph{autonomous}.}
 if there is a natural isomorphism\footnote{Of course, one can define the lax version, by dropping the isomorphism restriction on $\ka$. But as this is of no interest for this paper, we have chosen to work from the beginning with the strong version.} $\ka: SFX\to FSX$ such that the following diagrams commute:
\begin{align}
\label{eq1:lax_pres_dual} \vcenter{
\xymatrix{{\I} \ar[rr]^{f_0} \ar[d]_{\db F} & &  F{\I} \ar[d]^{F\db} \\
FX \ot SFX \ar[r]^{1\ot \ka} & FX \ot FSX \ar[r]^{f_2} &F(X \ot SX) 
}} 
\\ 
\label{eq2:lax_pres_dual}
\vcenter{
\xymatrix{ SFX \ot FX \ar[d]_{\e F} \ar[r]^{\ka \ot 1} & FSX \ot FX \ar[r]^{f_2} & F(SX \ot X) \ar[d]^{F\e} \\
{\I} \ar[rr]^{f_0} && F{\I} }
}
\end{align}

Dually, a monoidal functor $F:\C\to \D$ between right autonomous categories is \emph{right autonomous} if there is a natural isomorphism $\lambda:S'F\to FS'$ such that the following diagrams commute:
\begin{align}
\label{eq1':lax_pres_dual} \vcenter{
\xymatrix{{\I} \ar[rr]^{f_0} \ar[d]_{\db' F} & &  F{\I} \ar[d]^{F\db'} \\
S'FX \ot FX \ar[r]^{\la \ot 1} & FS'X \ot FX \ar[r]^{f_2} &F(S'X \ot X) 
}}  
\\
\label{eq2':lax_pres_dual}
\vcenter{
\xymatrix{ FX \ot S'FX \ar[d]_{\e' F} \ar[r]^{1 \ot \la} & FX \ot FS'X \ar[r]^{f_2} & F(X \ot S'X) \ar[d]^{F\e'} \\
{\I} \ar[rr]^{f_0} && F{\I} }}
\end{align}

\end{definition}

\begin{remark}\label{FMF pres dual}

Frobenius monoidal functors are autonomous \cite{day-pastro:Frob}: given two categories $\C $ and $\D$, say left autonomous, and $F:\C \to \D$ a Frobenius monoidal functor, the (left) autonomy of $F$ is witnessed by an (iso)morphism $\ka:SF\to FS$ given by the following composite:\footnote{The inverse of $\ka$ is  easily checked to be the morphism below:
\[\xymatrix@1@C=20pt{FSX \ar[r]^-{1 \ot \db} & FSX \ot FX \ot SFX \ar[r]^{f_2 \ot 1} & F(SX \ot X) \ot SFX \ar[r]^-{F\e \ot 1} & F{\I} \ot SFX \ar[r]^-{F_0\ot 1} & SFX}
\]
} 
\begin{align*}
\xymatrix@C=15pt{SFX \ar[r]^-{1\ot f_0} & SFX \ot F{\I} \ar[r]^-{1\ot F\db} & SFX \ot F(X \ot SX) \ar[r]^-{1 \ot F_2} & SFX \ot FX \ot FSX \ar[r]^-{\e \ot 1} & FSX }
\end{align*}

As strong monoidal functors are particular cases of Frobenius monoidal ones, we thus recover the well-known result that strong monoidal functors on left/right autonomous categories are left/right autonomous~\cite{js:braided}. 
\end{remark}

Now consider $F:\C \to \D$ a left autonomous functor. Because $S$ is strong monoidal, $SF$ becomes a monoidal functor $\C\to \D^{\op,\cop}$, while $F^\op S$ is a comonoidal one. The natural transformation $\ka$ respects the \linebreak (co)monoidal structure of these functors:

\begin{proposition}\label{prop:colax-lax}

Let $F:\C \to \D$ be an autonomous monoidal functor between left autonomous categories. Then $\ka:SF \to F^\op S:\C \to \D^{\op, \cop}$ is a monoidal-comonoidal natural transformation.

\end{proposition} 

\begin{proof}
That is, the following diagrams commute:
\begin{equation}
\label{eq:lax-colax}
\vcenter{
\xymatrix@C=25pt{SF{\I} \ar[rr]^{\ka} \ar[d]_{Sf_0} & & FS{\I}   & SF(Y \ot X)  \ar[d]_{Sf_{2}}  \ar[r]^{\ka} & FS(Y \ot X)   \\
S{\I} \ar[dr]_{s^{-1}_0} & & F{\I} \ar[u]_{Fs_0}  & S(FY \ot FX) \ar[d]_{s^{-1}_{2}} &  F(SX \ot SY)   \ar[u]_{Fs_2} \\
& {\I} \ar[ur]_{f_0} &  &SFX\ot SFY     \ar[r]^{\ka\ot \ka } & FSX\ot FSY \ar[u]_{f_2} 
}} 
\end{equation} 
Due to the large diagrams involved, the proof was deferred to~\ref{proof}. 
\end{proof}

\begin{remark}
The above proposition gives the major technical result upon which the paper relies.  We shall see in the sequel (Proposition~\ref{prop:lax-colax}) that in case both dualities exist on categories $\C $ and $\D$, a monoidal functor $F$ endowed with an isomorphism $\ka:SF \to F^\op S$ is (left) autonomous if and only if $\ka$ is a monoidal-comonoidal natural transformation. 
\end{remark}

Let now $F:\C \to \D$ be just a monoidal functor between autonomous categories. Using the contravariant equivalence between the left and right duality functors, one can see that natural transformations $\ka:SF\to F^\op S$ are in one-to-one correspondence with natural transformations $\lambda:S'F^\op \to FS'$ \cite{kellydoctrinal}, as indicated 
below: 
\begin{align}
\label{eq:mate-of-ka} 
& \lambda : \xymatrix@C=32pt{ S'F^\op \ar[r]^-{S'F^\op \beta^{-1} }& S'FSS'\ar[r]^{S'\ka S'} & S'SFS'\ar[r]^-{S'F \alpha^{-1}} & FS' }
\\
\label{eq:mate-of-la} 
& \ka : \xymatrix@C=35pt{ SF \ar[r]^-{SF\alpha^{-1}} & SFS'S \ar[r]^{S\la S} & SS'F^\op S \ar[r]^-{\beta^{-1}F^\op S} & F^\op S}
\end{align}
Notice that $\ka$ is an isomorphism if and only if $\la$ is so. 

\begin{proposition}\label{rem:leftaut=rightaut}

A monoidal functor $F:\C\to \D$ between autonomous categories is left autonomous if and only if it is right autonomous. More in detail, $\ka$ satisfies~\eqref{eq1:lax_pres_dual} if and only if $\lambda$ given by~\eqref{eq:mate-of-ka} satisfies~\eqref{eq1':lax_pres_dual}; dually, $\ka$ satisfies~\eqref{eq2:lax_pres_dual} if and only if~\eqref{eq2':lax_pres_dual} holds for $\la$.

\end{proposition}

\begin{proof} Again, this is deferred to~\ref{proof2}. 
\end{proof}

From the above proposition we see that we can talk without ambiguity about autonomous (monoidal) functors, without necessity to add the adjective "left" or "right", this being deduced from the context (i.e. whenever the categories involved are both left or both right autonomous).

\begin{proposition} \label{prop:lax-colax}

Let $F:\C\to \D$ be a monoidal functor between autonomous categories and $\ka:SF\to F^\op S$ a natural isomorphism, with $\la:S'F^\op \to FS'$ its mate, as in~\eqref{eq:mate-of-ka}. The following are equivalent:
\begin{enumerate}
\item $F$ is autonomous;
\item $\ka$ is a monoidal-comonoidal natural transformation;
\item $\la$ is a comonoidal-monoidal natural transformation. 
\end{enumerate}
\end{proposition}

\begin{proof} The implication \textit{(1)$\Longrightarrow$(2)} was proved in Proposition \ref{prop:colax-lax}. For its converse, assume first that~\eqref{eq:lax-colax} holds, and take $X=S'Y$. This proves~\eqref{eq1':lax_pres_dual}. Dually, making $Y=S'X$ in~\eqref{eq:lax-colax} proves~\eqref{eq2':lax_pres_dual}. Consequently, $F$ is autonoumous. 

The equivalence \textit{(1)$\Longleftrightarrow$(3)} follows similarly.
\end{proof}


\begin{remark} \label{klst}
Let $F:\C \to \D$ a monoidal functor between autonomous categories and $\ka:S F\to F^\op S$ a natural transformation, with mate $\la:S'F^\op \to FS'$. Then the diagrams below commute:
\begin{equation} \label{eq:ka-la}
\vcenter{\xymatrix{ S'F^\op S \ar[r]^{S'\ka} \ar[d]_{\la S} & S'SF \ar[d]^{ \alpha^{-1} F}\\
FS'S \ar[r]^-{F \alpha^{-1} } & F \\
}}
\qquad 
\vcenter{\xymatrix{SFS' \ar[d]_{\ka S'} \ar[r]^{S\la}  & SS'F^\op \ar[d]^{\beta^{-1}F} \\
F^\op SS' \ar[r]^-{F^\op \beta^{-1}} & F^\op }} 
\end{equation}
Denote by $\sigma $ and $\tau$ the common composites in the above diagrams; that is, 
\begin{align}
\label{eq:FtoSFS}
& \sigma = \alpha^{-1}F \circ S'\ka  :S'F^\op S \to F:\C \to \D \\
& \tau =F^\op \beta^{-1}\circ \ka {S'} : SFS' \to F^\op:\C^{\op, \cop} \to \D^{\op, \cop}
\end{align}
Then $\sigma $ is a comonoidal-monoidal natural transformations if and only if $\ka$ is monoidal-comonoidal. Similarly, $\tau$ is monoidal-comonoidal if and only if $\la$ is comonoidal-monoidal. Additionally, each of $\sigma,\tau$ is an isomorphism if and only if the other is, if and only if $\ka$ or $\la$ are so. \end{remark}

\begin{theorem} \label{thm:aut->frob}
Let $(F, f_2, f_0):\C\to \D$ be an autonomous monoidal functor between autonomous categories, with isomorphism $\ka:SF\to F^\op S$. Then there is a comonoidal structure $(F, F_2, F_0)$ on $F$ such that $(F, f_2, f_0, F_2, F_0):\C \to \D$ becomes a Frobenius monoidal functor. 
\end{theorem}

\begin{proof}
The comonoidal structure of $F$ can be obtained via the (iso)morphism $\sigma$ from the comonoidal structure of $S'F^\op S$; explicitly, $F_0:F{\I} \to {\I}$ and $F_{2, X, Y}:F(X \ot Y)\to FX \ot FY $ are the composites 
\begin{align} 
&\xymatrix{F{\I} \ar[r]^-{\sigma^{-1}} & S'F^\op S{\I} \ar[r]^{S'F^\op s_0} & S'F^\op {\I} \ar[r]^{S'f_0} & S'{\I} \ar[r]^{{s'}^{-1}_0}& {\I}}  \label{monoidal-str0} \\
&\vcenter{\xymatrix@C=10pt@R=10pt{F(X \ot Y) \ar[r]^-{\sigma^{-1}} & S'F^\op S(X \ot Y) \ar[r]^{S'F^\op s_2} & S'F^\op (SY \ot SX) \\
\ar[r]^-{S'f_2} & S'(F^\op SY \ot F^\op SX)
\ar[r]^-{{s'}^{-1}_2}& S'F^\op SX \ot S'F^\op SY \ar[r]^-{\sigma \ot \sigma} & FX \ot FY } }\label{monoidal-str2}
\end{align}
Alternatively, one can also obtain a comonoidal structure on $F$ using the isomorphism $\tau:SFS'\to F^\op $. There should be no surprise in the fact that  these two monoidal structures agree; see~\ref{proof3}, where subsequently we show that diagram~\eqref{eq1:Frob} commutes using the comonoidal structure obtained from $\sigma$. By duality, diagram~\eqref{eq2:Frob} will also commute using the comonoidal structure induced by $\tau$. 
\end{proof}

We can resume the results of this subsection as follows:

\begin{theorem}\label{cor:frob}
Let $(F, f_2, f_0):\C\to \D$ be a monoidal functor between autonomous categories. The following are equivalent:
\begin{enumerate}
\item There is a comonoidal structure $(F, F_2, F_0)$ on $F$ such that $F$ becomes a Frobenius monoidal functor.
\item $F$ is an autonomous functor.
\item There exists a \emph{monoidal-comonoidal} natural isomorphism $\ka: SF \to F^\op S:\C \to \D^{\op, \cop}$. 
\item There exists a \emph{comonoidal-monoidal} natural isomorphism $\la: S'F^\op  \to FS':\C^{\op, \cop}\to \D$.
\item There is a \emph{comonoidal-monoidal} isomorphism $\sigma: S'F^\op S\to F:\C \to \D$.
\item There is a \emph{monoidal-comonoidal} isomorphism $\tau: SFS'\to F^\op :\C^{\op, \cop} \to \D^{\op, \cop}$.
\end{enumerate}
\end{theorem}

Of course, a completely dual result holds for $F$ being a comonoidal functor. We leave the reader to fill-in the details.

\subsection{When linear functors are Frobenius monoidal} \label{subsect:lin funct pres dual}

As mentioned in Section~\ref{sect:lin funct mon cat}, monoidal categories are degenerate linearly distributive categories, with the two tensor products identified. Adding \emph{negations} leads to autonomous categories. Recall from~\cite{cockett-seely:lf} that a linear functors $(R,L)$ between (degenerate) linearly distributive categories with negations is completely determined by its monoidal part $R$ and by a coherent isomorphism between its "de Morgan duals" $SRS'\cong S'RS$, which will provide the comonoidal part $L$. We shall corroborate this with the results of the previous subsection, in order to see that that for a linear functor between autonomous categories, one component is Frobenius monoidal if and only if the other is, if and only if there is an isomorphism between them compatible with the four (co)strengths, if and only if there is a monoidal-comonoidal isomorphism between them.  

\bigskip

Consider thus $(R,L):\C \to \D$ a linear functor between \emph{left} autonomous categories. Then for each object $X$ in $\C$, it is easy to check that $RSX$ becomes a left dual to $LX$, with unit and counit 
\begin{align*}
&\mathfrak d{:}\xymatrix{{\I} \ar[r]^{r_0} & R{\I} \ar[r]^-{R\db} & R(X \ot SX) \ar[r]^-{\nrr} & LX \ot RSX} \\
&\mathfrak e {:} \xymatrix{ RSX \ot LX \ar[r]^{\nrl} & L(SX \ot X) \ar[r]^-{Le} & L{\I} \ar[r]^{L_0} & {\I}}
\end{align*}
hence it induces a natural isomorphism $\Omega:RS \to SL$ by
\begin{equation}
\label{def Omega}
\Omega: \xymatrix{RSX \ar[r]^-{1\ot \db} & RSX \ot LX \ot SLX \ar[r]^-{\mathfrak e \ot 1} & SLX}
\end{equation}
such that:
\begin{align}
& \vcenter{\xymatrix{ {\I}  \ar@{=}[d] \ar[r]^{r_0} & R{\I} \ar[r]^-{R\db} & R(X \ot SX) \ar[r]^{\nrr} & LX \ot RSX \ar[d]^{1\ot \Omega} \ar@{<-}`u[lll]`[lll]_{\mathfrak d}[lll]\\
{\I} \ar[rrr]_{\db} & & & LX \ot SLX }
}
\\
& \vcenter{
\xymatrix{ RSX \ot LX \ar`u[rrr]`[rrr]^{\mathfrak e}[rrr] \ar[r]^{\nrl} \ar[d]_-{\Omega\ot 1} & L(SX \ot X) \ar[r]^{L\e} & L{\I} \ar[r]^{L_0} & {\I} \ar@{=}[d]  \\ 
SLX \ot LX \ar[rrr]_{\e} & && {\I}  
}
}
\end{align}

\begin{lemma}
The isomorphism above is a comonoidal natural transformation $\Omega:R^\op S \to SL:\C \to \D^{\op, \cop}$ , in the sense that it satisfies
\[
\vcenter{\xymatrix{RSY \ot RSX \ar[d]_{r_2} \ar[r]^{\Omega \ot \Omega} & SLY \ot SLX \ar[d]^{s_2} \\
R(SY \ot SX) \ar[d]_{Rs_2} & S( LX \ot LY) \ar[d]^{SL_2} \\
RS(X \ot Y) \ar[r]^{\Omega} & SL(X \ot Y) 
}} \qquad 
\vcenter{
\xymatrix{
& {\I} \ar[dl]_{r_0} \ar[dr]^{s_0} & \\
R{\I} \ar[d]_{Rs_0} && S{\I} \ar[d]^{SL_0} \\
RS{\I} \ar[rr]^{\Omega} && SL{\I} }}
\]
\end{lemma}

\begin{proof}

Both diagrams above are diagrams between objects with left duals; they commute if and only if the diagrams between the corresponding left duals do so; but the latter are 
\[
\xymatrix{LX \ot LY \ar@{=}[r] & LX \ot LY  && {\I} & \\
L(X \ot Y) \ar[u]^{L_2} \ar@{=}[r] &L(X \ot Y) \ar[u]_{L_2}  & L{\I} \ar@{=}[rr] \ar[ur]^{L_0} && L{\I} \ar[ul]_{L_0}
}
\] 
and obviously commute. 
\end{proof}

Similarly, one can show the existence of a natural isomorphism $\Psi: L^\op S\cong SR:\C \to \D^{\op, \cop}$ of monoidal functors. 

\bigskip

Assume now that both the domain and codomain categories are autonomous, therefore \emph{both right and left duals} exist. Then we obtain a comonoidal isomorphism 
\begin{align}\label{Omega}
L\overset{\alpha L}{\longrightarrow} S'SL \overset{S'\Omega}{\longrightarrow} S'R^\op S:\C \to \D
\end{align}
and a monoidal isomorphism 
\[
L^\op  \overset{L^\op \beta}{\to} L^\op SS' \overset{\Psi S'}{\to} SRS':\C^{\op, \cop} \to \D^{\op, \cop}
\]
Consequently, the structural comonoidal morphisms $L_2$ and $L_0$ can be explicitly described only in terms of $r_2, r_0$ and $\Omega$, respectively $\Psi$. And so do the corresponding costrenghts, 
 in the sense that:
\begin{align}
\label{nrr=right-closed}
& \vcenter{\xymatrix{R(X \ot Y) \ar[d]_-{\db' \ot 1} \ar[r]^\nrr & LX \ot RY \ar[d]^-{\alpha L \ot 1} \\
S'RSX \ot RSX \ot R(X \ot Y) \ar[d]_{1 \ot r_2} & S'SLX \ot RY \ar[d]^{S'\Omega\ot 1} \\
S'RSX \ot R(SX \ot X \ot Y) \ar[r]^-{1 \ot R(\e \ot 1)} & S'RSX \ot RY }}
\\
\label{nrl=left-closed}
& \vcenter{\xymatrix{R(X \ot Y) \ar[d]_-{1\ot \db} \ar[r]^\nlr & RX \ot LY \ar[d]^-{1 \ot L\beta} \\
R(X \ot Y) \ot RS'Y \ot SRS'Y \ar[d]_{r_2\ot 1 } & RX \ot LSS'Y \ar[d]^{1 \ot \Psi S'} \\
R(X \ot Y \ot S'Y) \ot SRS'Y \ar[r]^-{R(1 \ot \e')\ot 1 } & RX \ot SRS'Y 
}}
\end{align}
By duality, there are two more such diagrams involving the strengths $\nrl, \nll$ and the mates of the isomorphisms $\Psi$ and $\Omega$ with respect to the adjunction $S \dashv S'$. 

\begin{proposition}\label{prop:when lin is frob}

Let $(R, L):\C \to \D$ a linear functor between autonomous categories. Then the following are equivalent:

\begin{enumerate}

\item There is a natural isomorphism $\omega:R\cong L$ compatible with the linear structure, in the sense that the following four conditions hold:
\begin{align}
\label{omega}& \vcenter{
\xymatrix{RX \ot RY \ar[d]_{r_2} \ar[r]^{1\ot \omega} & RX \ot LY \ar[d]^{\nrl} \\
R(X \ot Y) \ar[r]^\omega & L(X \ot Y) }}
& \begin{cases}
\nrl \circ (1 \ot \omega ) = \omega \circ r_2 \\ 
\nll\circ (\omega \ot 1) = \omega \circ r_2 \\
(1 \ot \omega) \circ  \nrr  = L_2 \circ \omega \\ 
(\omega \ot 1) \circ  \nlr= L_2  \circ \omega 
\end{cases}
\end{align}

\item $R$ is a Frobenius monoidal functor;

\item $R$ is autonomous;

\item $L$ is a Frobenius monoidal functor;

\item $L$ is autonomous;

\item There is a monoidal-comonoidal natural isomorphism $R \cong L$.

\end{enumerate}

\end{proposition}

\begin{proof}
The equivalence \textit{2.}$\Leftrightarrow$ \textit{3.} follows from Theorem~\ref{cor:frob}, while \textit{4.}$\Leftrightarrow$\textit{5.} can be obtained by duality. 

\textit{1.}$\Rightarrow $\textit{2.} First, notice that $R$ inherits a comonoidal structure from $L$ via the isomorphism $\omega$. Then its monoidal and comonoidal structures obey the Frobenius relations~\eqref{eq2:Frob},~\eqref{eq1:Frob}. For example, the commutativity of~\eqref{eq2:Frob} is proved below, while for the other one it follows similarly: 
\[
\xymatrix@C=15pt{
RX \ot R(Y\ot Z) \ar@{}[drrr]|{\eqref{lf4}} \ar[dr]^{1 \ot \nrr} \ar[d]_{1 \ot \omega} \ar[rrr]^{r_2} &&& R(X \ot Y \ot Z) \ar[dl]_{\nrr} \ar[d]^{\omega}  \\
RX \ot L(Y \ot Z) \ar[d]_{1 \ot L_2} \ar@{}[r]|{\eqref{omega}} & RX \ot LY \ot RZ \ar[dl]^{1 \ot 1 \ot \omega} \ar[r]^{\nrl \ot 1} & L(X\ot Y) \ot RZ \ar@{}[r]|{\eqref{omega}}  \ar[dr]_{1 \ot \omega} & L(X \ot Y \ot Z) \ar[d]^{L_2} \\
RX \ot LY \ot LZ \ar[d]_{1 \ot \omega^{-1}\ot \omega^{-1} } &&& L(X \ot Y) \ot LZ \ar[d]^{\omega^{-1} \ot \omega ^{-1} } \\
RX \ot RY \ot RZ \ar@{}[uurrr]|{\eqref{omega}}  \ar[rrr]^{r_2\ot 1} && & R(X \ot Y ) \ot RZ }
\]
Notice that only two of the relations~\eqref{omega} are needed to prove that $R$ is a Frobenius monoidal functor. 

\textit{1.}$\Rightarrow $\textit{4.} is proved similarly, using the other two relations from~\eqref{omega}. 

\textit{3.}$\Rightarrow$\textit{1.} The functor $R$ being autonomous, there is a comonoidal-monoidal isomorphism $\sigma:S'R^\op S\to R$~\eqref{eq:FtoSFS}, which by precomposition with the comonoidal isomorphism $L\overset{\alpha L}{\to} S'SL \overset{S'\Omega}{\to} S'R^\op S$ from~\eqref{Omega} produces a comonoidal-monoidal isomorphism $L\to R$, whose inverse we shall denote by $\omega$. It is enough to check only one of the relations~\eqref{omega}, the other three being obtained by passage to $\C^\op$, $\C^\cop$ and respectively $\C^{\op, \cop}$. For example, the third relation $(1 \ot \omega) \circ  \nrr  = L_2 \circ \omega$ follows from the commutative diagram below: 
\[
\xymatrix@C=7pt{R(X \ot Y) \ar@<+1.25pt>@{-}`r[rrrddd][rrrddd] \ar@<-1.25pt>@{-}`r[rrrddd][rrrddd] \ar[ddddd]_{\nrr} \ar[dr]^-{\db'\ot 1} && \\
& S'RSX \ot RSX \ot R(X\ot Y) \ar[dr]^{1 \ot r_2} &&\\
& & S'RSX \ot R(SX \ot X \ot Y) \ar[d]^{1 \ot R(\e \ot 1)} \ar@{}[dll]|{\eqref{nrr=right-closed}} &&\\
& & S'RSX \ot RY \ar[d]^{\sigma^{-1} \ot 1}  & R(X \ot Y) \ar[d]^{\omega} &&\\
&S'SLX \ot RY \ar[ur]^{S'\Omega\ot 1} & RX \ot RY \ar[ur]^{r_2} \ar[dr]^{\omega \ot \omega} & L(X \ot Y) \ar[d]^{L_2} &&\\
LX \ot RY \ar[ur]^{\alpha L \ot 1} \ar[urr]_{\omega^{-1}\ot 1} \ar[rrr]_{1 \ot \omega} & & & LX \ot LY &&
}
\]
where the bottom right triangle commutes since $\omega$ is a monoidal-comonoidal isomorphism, and the commutativity of the top right diagram is a consequence of the definition of $\sigma$ in~\eqref{eq:FtoSFS}, and of~\eqref{eq1:lax_pres_dual}. The implication \textit{5.}$\Rightarrow$\textit{1.} is proved similarly. 

Finally, the equivalences \textit{3.}$\Leftrightarrow$ \textit{6.} and \textit{5.}$\Leftrightarrow$ \textit{6.} are obtained using the isomorphism~\eqref{Omega} and  Theorem~\ref{cor:frob}. 
\end{proof}


\begin{appendices}


\section{Monoidal and linear functors}\label{app:prelim}


\subsection{Monoidal categories and functors} \label{app:monoidal}

In this appendix and the subsequent two ones, we recall the notions of monoidal functors, linear functors (between monoidal categories), and of autonomous categories. Because mainly all our results are proofs are based on hypotheses and proofs by commutative diagrams, we needed to fix notations and to put labels on all necessary equations and commutative diagrams.

\medskip

A monoidal functor $F:\C\to \D$ between monoidal categories is a functor endowed with a natural transformation $f_{2, X,Y}:FX \ot FY \to F(X\ot Y)$ and a morphism $f_0:{\I}\to F{\I}$ such that
\begin{eqnarray}
\label{lax-functor0}& \vcenter{\xymatrix{ FX \ar@<-1.25pt>@{-}`r[dr][dr] \ar@<+1.25pt>@{-}`r[dr][dr] \ar[d]_{f_0 \ot 1} & \\
 F{\I} \ot FX \ar[r]^-{f_{2}} & FX}}
\quad
\vcenter{\xymatrix{FX \ar@<-1.25pt>@{-}`r[dr][dr]  \ar@<+1.25pt>@{-}`r[dr][dr] \ar[d]_{1\ot f_0} & \\
FX\ot F{\I} \ar[r]^-{f_{2}} & FX 
}}
\\
\label{lax-functor2}& \vcenter{\xymatrix@C=36pt{FX\ot FY\ot FZ \ar[r]^{f_{2}\ot 1} \ar[d]_{1\ot f_{2}} & F(X\ot Y) \ot FZ \ar[d]^{f_{2}} \\
FX\ot F(Y\ot Z) \ar[r]^{f_{2}} & F(X\ot Y\ot Z)
}}
\end{eqnarray}
Dually, a functor $F:\C \to \D$ is comonoidal if there is a natural transformation $F_{2}:F(X\ot Y) \to FX \ot FY$ and a morphism $
F_0:F{\I} \to {\I}$ such that 
\begin{eqnarray}\label{colax_functor}
&\label{colax_functor0} \vcenter{
\xymatrix{FX \ar[r]^{F_{2}} \ar@<-1.25pt>@{-}`d[dr][dr] \ar@<+1.25pt>@{-}`d[dr][dr] & F{\I} \ot FX \ar[d]^{F_0 \ot 1} \\
&FX
}
} 
\quad 
\vcenter{ \xymatrix{FX \ar@<-1.25pt>@{-}`d[dr][dr] \ar@<+1.25pt>@{-}`d[dr][dr] \ar[r]^{F_{2}}  & FX \ot F{\I} \ar[d]^{1 \ot F_0} \\
& FX
}
} 
\\
& \label{colax_functor2} \vcenter{
\xymatrix@C=36pt{ F(X\ot Y \ot Z) \ar[d]_{F_{2}} \ar[r]^{F_{2}} & FX \ot F(Y\ot Z) \ar[d]^{1\ot F_{2}} \\
F(X\ot Y) \ot FZ \ar[r]^{F_{2}\ot 1} & FX\ot FY \ot FZ }
}
\end{eqnarray}

To emphasize the difference, the structural morphisms for monoidal functors will be denoted by small letters, while in case of comonoidal functors we shall employ capital letters.

A (co)monoidal functor is called strong monoidal is its structural morphisms are isomorphisms, and small letters will be employed to denote them. 

\medskip

A natural transformation $\alpha:F\to G:\C\to \D$ between monoidal functors is monoidal if it satisfies 
\begin{equation}\label{lax-nat}
\vcenter{
\xymatrix@C=38pt{FX \ot FY \ar[d]_{f_{2}} \ar[r]^{\alpha \ot \alpha} & GX\ot GY \ar[d]^{g_{2}} \\
F(X\ot Y) \ar[r]^{\alpha} & G(X\ot Y)} } 
\quad 
\vcenter{\xymatrix{ & {\I} \ar[dr]^{g_0} \ar[dl]_{f_0} & \\ 
F{\I} \ar[rr]^{\alpha} & & G{\I} }}  
\end{equation}

Dually, a comonoidal natural transformation $\alpha:F\to G:\C\to \D$ between comonoidal functors has to satisfy 
\begin{equation}
\vcenter{
\xymatrix@C=38pt{
F(X\ot Y) \ar[r]^{\alpha} \ar[d]_{F_{2}} & G(X\ot Y) \ar[d]^{G_{2}} \\
FX \ot FY \ar[r]^{\alpha \ot \alpha} & GX \ot GY}}
\quad 
\vcenter{
\xymatrix{F{\I} \ar[rr]^{\alpha} \ar[dr]_{F_0} &&G{\I} \ar[dl]^{G_0} \\
& {\I}& }}
\end{equation}


\subsection{Linear functors between monoidal categories}\label{linear}

A \emph{linear functor} between monoidal categories $\C$ and $\D$ consists of a pair of functors $R, L:\C\to \D$ , such that $R$ is monoidal and $L$ is comonoidal, with structure maps 
\begin{align*} 
r_{2}:R X \ot R Y \to R (X\ot Y), \qquad & r_0:{\I}\to R {\I}\\
L_{2}:L(X \ot  Y) \to L X \ot L Y, \qquad & L_0:L {\I}\to {\I}
\end{align*} 
together with four natural transformations
\begin{align*} 
&\nu_R^r:R (X\ot Y)\to L X\ot R Y,  \qquad \nu_R^l:R (X\ot Y)\to R X \ot L Y \\
&\nu_L^r:R X \ot L Y \to L (X\ot Y), \qquad \nu_L^l: L X \ot R Y \to L (X\ot Y) 
\end{align*} 
subject to the several conditions listed below, grouped such that a relation of each group is illustrated by a commutative diagram, from which the other relations belonging to the same group can be easily obtained as follows: the passage $R/L$ corresponds to a move to $\C^\op=(\C^\op, \ot)$, while the passage $r/l$ is obtained for $\C^\mathsf{cop}=(\C, \ot^\mathsf{rev})$. Finally, both changes become simultaneously available in $\C^{\op, \mathsf{cop}}=(\C^\op, \ot^{\mathsf{rev}})$. 
\begin{align}
\label{lf1}&
\quad \vcenter{
\xymatrix@C=55pt@R=35pt{
RX \ar[d]_{\nu^r_R} \ar@<-1.25pt>@{-}`r[dr][dr] \ar@<+1.25pt>@{-}`r[dr][dr] & \\
L{\I} \ot RX \ar[r]^-{L_0\ot 1} & RX 
}
}
& &
\begin{cases} 
(L_0\ot 1) \circ \nu^r_R =1 \\
(1\ot L_0) \circ \nu^l_R =1 \\
\nu^r_L \circ (r_0 \ot 1) = 1 \\
\nu^l_L \circ (1\ot r_0) = 1
\end{cases} 
\\
\label{lf2} & 
\vcenter{
\xymatrix{
R(X\ot Y \ot Z) \ar[r]^{\nu^r_R} \ar[d]_{\nu^r_R} & LX \ot R(Y\ot Z) \ar[d]^{1\ot \nu^r_R} \\
L(X\ot Y) \ot RZ \ar[r]^{L_2\ot 1} & LX \ot LY \ot RZ 
}
} 
& &
\begin{cases}
(1\ot \nu^r_R) \circ \nu^r_R = (L_2 \ot 1) \circ \nu^r_R \\
(\nu^l_R\ot 1) \circ \nu^l_R = (1\ot L_2) \circ \nu^l_R \\
\nu^r_L \circ (1\ot \nu^r_L) = \nu^r_L \circ (r_2 \ot 1) \\
\nu^l_L \circ (\nu^l_L\ot 1) = \nu^l_L \circ (1\ot r_2 ) 
\end{cases}
\\
\label{lf3} & 
\vcenter{\xymatrix{
R(X\ot Y \ot Z) \ar[r]^{\nu^r_R} \ar[d]_{\nu^l_R} & LX \ot R(Y\ot Z) \ar[d]^{1\ot \nu^l_R} \\
R(X\ot Y) \ot LZ \ar[r]^{\nu^r_R \ot 1} & LX \ot RY \ot LZ}
}
& &
\begin{cases}
(1\ot \nlr) \circ \nrr = (\nrr \ot 1) \circ \nlr \\
\nrl \circ (1 \ot \nll) = \nll \circ (\nrl \ot 1) \\
\end{cases} 
\\
\label{lf4} & 
\vcenter{
\xymatrix{RX \ot R(Y \ot Z) \ar[r]^{1\ot \nrr} \ar[d]_{r_2} & RX \ot LY \ot RZ \ar[d]^{\nrl \ot 1} \\
R(X \ot Y \ot Z) \ar[r]^{\nrr} & L(X\ot Y) \ot RZ} }
& &
\begin{cases}
(\nrl \ot 1) \circ (1 \ot \nrr) = \nrr \circ r_2 \\
(1 \ot \nll) \circ (\nlr\ot 1) = \nlr \circ r_2 \\
(1 \ot \nrl) \circ (\nrr \ot 1) = L_2 \circ \nrl \\
(\nll \ot 1) \circ (1 \ot \nlr) = L_2 \circ \nll 
\end{cases}
\\
\label{lf5}& 
\vcenter{
\xymatrix{ R(X\ot Y) \ot RZ \ar[r]^{\nrr \ot 1} \ar[d]^{r_2} & LX \ot RY \ot RZ \ar[d]^{1 \ot r_2} \\
R(X \ot Y \ot Z) \ar[r]^{\nrr} & LX\ot R(Y \ot Z) 
}
}
&&
\begin{cases}
(1 \ot r_2 ) \circ (\nrr\ot 1) = \nrr \circ r_2 \\
(r_2 \ot 1) \circ (1\ot \nlr) = \nlr \circ r_2 \\
(\nrl \ot 1) \circ (1 \ot L_2) =L_2 \circ \nrl \\
(1 \ot \nll) \circ (L_2 \ot 1) =L_2 \circ \nll 
\end{cases}
\end{align}

\subsection{Autonomous categories}\label{app:aut} 

Let $\C$ be a monoidal category. A \emph{left dual} of an object $X$ in $\C$ consists of another object $SX$, together with a pair of arrows $\db:{\I}\to X\ot SX$, $\e:SX\ot X\to {\I}$, satisfying the relations \eqref{left-dual}.  
A monoidal category is called \emph{left autonomous} if each object has a left dual.

\begin{proposition}
Consider two objects $X,Y$ with left duals and two morphisms $f:X\to Y$, $g:SY\to SX$. Then $g=Sf$ as in~\eqref{transpose} if and only if one of the equivalent conditions below hold: 
\begin{align}
\label{eq2:dinat} 
& \vcenter{
\xymatrix{
{\I} 
\ar[r]^-{\db} 
\ar[d]_{\db} 
& 
X\ot SX 
\ar[d]^{f\ot 1} 
\\ 
Y\ot SY 
\ar[r]_{1\ot g} 
& 
Y\ot SX 
}
}
&
\vcenter{
\xymatrix{
SY \ot X 
\ar[r]^-{g\ot 1} 
\ar[d]_{1\ot f} 
& 
SX\ot X 
\ar[d]^{\e} 
\\ 
SY\ot Y 
\ar[r]_{\e} 
& 
{\I}
}
}
\end{align}
\end{proposition}

For composable arrows $\xymatrix@1{X\ar[r]^f& Y\ar[r]^g& Z}$, between objects with left duals, one can easily check the relation $S(g\circ f)=Sf\circ Sg$. Also, $S 1_X=1_{SX}$ holds for each object $X$ having a left dual. Assuming a choice of duals in a left autonomous category $\C$, we obtain a contravariant functor $S$ on $\C$, such that:
\begin{proposition}
The functor $S:\C \to \C^{\op, \cop}$ is strong monoidal, with monoidal structure  $s_0:{\I} \to S{\I}$, $s_{2}:SX \ot SY \to S(Y \ot X)$ given by the unique (iso)morphisms making the diagrams below commute: 
\begin{align}
&\label{eq:s2db}\vcenter{\xymatrix@C=35pt{{\I} 
\ar`d[drr][drr]^-{\db} \ar[r]^-{\db} & X\ot SX \ar[r]^-{1\ot \db\ot 1} & X\ot Y \ot SY \ot SX \ar[d]^{1\ot 1\ot s_2} \\
& & X \ot Y \ot S(X\ot Y) 
}} \\
&\label{eq:s2ev}\vcenter{\xymatrix@C=35pt{SY \ot SX \ot X \ot Y \ar[d]_{s_2\ot 1\ot 1} \ar[r]^-{1\ot \e \ot 1} & SY \ot Y \ar[r]^-{\e} & {\I} \\
S(X\ot Y) \ot X \ot Y \ar`r[urr]^-{\e}[urr] & }}
\end{align}
\end{proposition}

In the sequel, whenever we refer to a left autonomous category, we shall implicitly assume a choice of left duals, hence a contravariant left duality functor $S:\C \to \C^{\op, \cop}$. 
\bigskip


Let now again consider $\C$ to be just a monoidal category. A \emph{right dual} of an object $X$ is an object $S'X$, together with arrows $\db':{\I}\to S'X\ot X$, $\e':X\ot S'X\to {\I}$, satisfying the relations~\eqref{eq:rightdual}. A monoidal category is called \emph{right autonomous} if each object has a right dual.

An autonomous category is both left and right autonomous. The contravariant pair of adjoint equivalences $S\dashv S':\C^{\op,\cop} \to \C$ has unit and counit given by the (monoidal) natural isomorphisms
\begin{align*}
&\alpha:\xymatrix@C=30pt{ X \ar[r]^-{\db' S\ot 1} & S'SX\ot SX \ot X \ar[r]^-{1 \ot \e} & S'SX}
\\
&\beta:\xymatrix@C=30pt{ X \ar[r]^-{1\ot \db S'} & X\ot S'X \ot SS'X \ar[r]^-{\e'\ot 1} & SS'X}
\end{align*}
Notice that we can relate $\db', \e'$ and $\db,\e$ as follows:
\begin{equation}\label{d'-beta-d}
\vcenter{\xymatrix{{\I} \ar[r]^-{\db'} \ar[dr]_{\db {S'} }& S'X \ot X \ar[d]^{1\ot \beta} \\
& S'X \ot SS'X }} 
\quad \mbox{ and } \quad
\vcenter{\xymatrix{X\ot S'X \ar[d]_{\beta \ot 1} \ar[r]^-{\e'} & {\I} \\
SS'X \ot S'X \ar[ur]_{\e {S'}} & }}
\end{equation} 
Also, by doctrinal adjunction \cite{kellydoctrinal}, the (strong) monoidal structure of $S'$ is the mate of the (strong) monoidal structure of $S$, in the sense that $s'_2$ and $s'_0$ are the pastings
\begin{align}
\label{eq:S2doctrinalS'2}
&\xymatrix@C=17pt{S'X\ot S'Y \ar[r]^-{\alpha} & S'S(S'X \ot S'Y) \ar[r]^-{S's_2} & S'(SS'Y \ot SS'X) \ar[r]^-{\beta \ot \beta} & S'(Y\ot X) } \\
&\xymatrix@C=17pt{{\I} \ar[r]^-\alpha & S'S{\I} \ar[r]^-{S's_0} & S'{\I}} 
\end{align}

\section{Proofs}\label{app:proofs}


\subsection{Proof of Proposition~\ref{prop:colax-lax}}\label{proof}

The commutativity of the first diagram, involving the unit object ${\I}$ and the corresponding morphism $f_0:\I \to F\I$, is shown below:
\[
\xymatrix@C=25pt{SF{\I} \ar[rrr]^{\ka} \ar[d]_{Sf_0} \ar[dr]^{1 \ot f_0} & & & FS{\I} \ar[dl]_{1 \ot f_0}  \ar@{=}[d]  \\
S{\I} \ar@{}[r]|-{\eqref{eq2:dinat}} \ar[dr]_{\e_{\I} = s^{-1}_0} 
& SF{\I} \ot F{\I} \ar@{}[ur]|{(M)} \ar[d]^{\e} \ar[r]^{\ka \ot 1} \ar@{}[dr]|{\eqref{eq1:lax_pres_dual}} 
& FS{\I} \ot F{\I} \ar[r]^-{f_2} \ar@<-1.3ex>@{}[ur]|{\eqref{lax-functor0}} 
& FS{\I} \\
& {\I} \ar[r]_{f_0}  & F{\I}  \ar[ur]_{Fs_0=F(\e_{\I}^{-1})}
}
\]

In order to prove that the second diagram also commutes, we need the following two lemmas:

\begin{lemma}
Let $F:\C \to \D$ a monoidal functor, together with a natural isomorphism $\ka:SF \to FS$. Then: 
\begin{enumerate}
\item Diagram~\eqref{eq1:lax_pres_dual} commutes for all $X\in \C$  if and only if the 
diagram below does so, for each pair of objects $X, Y \in \C$ :
\begin{align}
\label{eq1:ka-db}
\vcenter{\xymatrix{FY \ar`d[drr][drr]_{F(1 \ot \db)} \ar[r]^-{1 \ot \db}& FY \ot FX \ot SFX \ar[r]^{f_2 \ot \ka} & F(Y \ot X) \ot FSX \ar[d]^{f_2} \\
&& F(Y \ot X \ot SX) 
}}
\end{align}
\item Dually, diagram~\eqref{eq2:lax_pres_dual} commutes if and only if the next diagram also commutes:
\begin{align}
\label{eq2:ka-ev}
\vcenter{
\xymatrix{SFX \ot FX \ot FY \ar`d[drr][drr]_{\e \ot 1}  \ar[r]^{\ka\ot f_2} & FSX \ot F(X \ot Y) \ar[r]^{f_2} & F(SX \ot X \ot Y) \ar[d]^{F(\e \ot 1)} \\ && FY
}}
\end{align}

\end{enumerate}
\end{lemma}

\begin{proof}
First, the diagram below proves that~\eqref{eq1:lax_pres_dual} implies ~\eqref{eq1:ka-db}:
\[
\label{eq:ka-ev}
\xymatrix{FY \ar@{-}@<-1.25pt>`d[dddr][dddr] \ar@{-}@<+1.25pt>`d[dddr][dddr] \ar[dr]^{1 \ot f_0} 
\ar[r]^-{1 \ot \db}& FY \ot FX \ot SFX \ar[rr]^{f_2 \ot \ka} \ar[dr]^{1 \ot 1 \ot \ka} & & F(Y \ot X) \ot FSX \ar[ddd]^{f_2} \\ 
&FY \ot F{\I} \ar@{}[dl]|{\eqref{lax-functor0}} \ar[dd]^{f_2} \ar@{}[r]|{\eqref{eq1:lax_pres_dual}} \ar[dr]^{1 \ot F\db} &FY \ot FX \ot FSX \ar@{}[u]|{(M)} \ar[ur]^{f_2 \ot 1} \ar[d]^{1 \ot f_2} & \\ 
& &FY \ot F(X \ot SX) \ar@{}[ur]|{\eqref{lax-functor2}} \ar[dr]^{f_2} & \\
&FY \ar@{}[ur]|{(N)} \ar[rr]^{F(1 \ot \db)} & & F(Y \ot X \ot SX) 
}
\]
To show the converse, take $Y={\I}$ in~\eqref{eq1:ka-db} and use~\eqref{lax-functor0}. 

The equivalence between~\eqref{eq2:lax_pres_dual} and~\eqref{eq2:ka-ev} follows by duality. 
\end{proof}

\begin{lemma}
Let $F:\C \to \D$ a left autonomous functor and denote by $f^{(\ot) 4}_2:FX \ot FY \ot FZ \ot FW \to F(X \ot Y \ot Z \ot W)$ the iteration of $f_2$ to a tensor product of four factors. Then the following diagrams commute: 
\begin{align}
\label{ka-db-ot}
&\vcenter{\xymatrix{{\I} \ar[ddd]_{f_0} \ar[r]^-\db &FX \ot SFX \ar[r]^-{1 \ot \db \ot 1}  &FX \ot FY \ot SFY \ot SFX \ar[d]^{1 \ot 1 \ot \ka \ot \ka} \\
&& FX \ot FY \ot FSY \ot FSX \ar[d]^{f_2^{(\ot) 4}} \\
&  & F(X \ot Y \ot SY \ot SX) \ar[d]^{F(1 \ot 1\ot s_2)} &  \\
F{\I} \ar[rr]^{F\db} &&F(X \ot Y \ot S(X \ot Y))} }
\\
\label{ka-ev-ot}
&\vcenter{\xymatrix{S(FX \ot FY) \ot FX \ot FY \ar[rr]^\e  \ar[d]_{s_2^{-1}\ot 1 \ot 1} & &  {\I}\ar[ddd]^{f_0}  \\
SFY \ot SFX \ot FX \ot FY \ar[d]_{\ka \ot \ka \ot 1\ot 1} & & \\
FSY \ot FSX \ot FX \ot FY \ar[d]_{f_2^{(\ot)4}} && \\
F(SY \ot SX \ot X \ot Y) \ar[r]^-{F(1 \ot \e \ot 1)} & F(SY \ot Y) \ar[r]^-{F\e} & F{\I}
}}
\end{align}
\end{lemma}

\begin{proof}
The commutativity of~\eqref{ka-db-ot} is proved below. The second one follows analogously. 
\[
\xymatrix{{\I} \ar@{}[ddr]|{\eqref{eq1:lax_pres_dual}} \ar[r]^\db \ar[dd]_{f_0} & FX \ot SFX \ar[ddr]_{F(1 \ot \db) \ot 1} \ar[d]_{1 \ot \ka} \ar[r]^-{1 \ot \db \ot 1} & FX \ot FY \ot SFY \ot SFX \ar@{}[dl]|{\eqref{eq1:ka-db}} \ar[d]^{f_2 \ot \ka \ot 1} \\
&FX \ot FSX \ar@{}[dr]|{(M)} \ar[ddr]_{F(1 \ot \db) \ot 1} \ar[d]_{f_2} & F(X \ot Y)\ot FSY \ot SFX \ar[d]^{f_2 \ot 1}  & \\
F{\I} \ar[dd]^{F\db} \ar[r]^{F\db} & F(X \ot SX) \ar@{}[dr]|{(M)} \ar@{}[ddl]|{\eqref{eq:s2db}} \ar[ddr]_{F(1 \ot \db \ot 1)}& F(X \ot Y \ot SY) \ot SFX \ar[d]^{1 \ot \ka}  \\
 && F(X \ot Y \ot SY) \ot FSX \ar[d]^{f_2} \\
F(X \ot Y \ot S(X \ot Y)) \ar[rr]^{F(1 \ot 1 \ot s_2^{-1})}  && F(X \ot Y \ot SY \ot SX) 
}
\]
\end{proof}

Now we are ready to continue with the proof of Proposition~\ref{prop:colax-lax}.
\medskip

The second diagram of~\eqref{eq:lax-colax} is shown to commute by diagram~\eqref{largediagram} on page \pageref{largediagram}. To save space,  we have omitted the subscript indices and wrote $X{\cdot} Y $ instead of $X\ot Y$. Notice that the composite on the right side of~\eqref{largediagram} is the identity, because: 
\[
\resizebox{12cm}{!}{
\xymatrix@C=30pt@R=18pt{FS(Y \ot X) \ar[r]^-{1 \ot f_0} & FS(Y \ot X) \ot F{\I} \ar[r]^-{1 \ot F\db} \ar[d]^{f_2} \ar@{}[dr]|{(N)}  & FS(Y \ot X) \ot F( Y \ot X \ot S(Y \ot X)) \ar[d]^{f_2} \ar[r]^{1 \ot F( 1 \ot 1\ot s_2^{-1} )} & FS(Y \ot X) \ot F(Y \ot X \ot SX \ot SY) \ar[d]^{f_2} 
\\
&FS(Y \ot X) \ar[r]^-{F(1 \ot \db)} \ar@{=}[ddrr] \ar@{=}[ul] & F(S(Y \ot X) \ot Y \ot X \ot S(Y \ot X)) \ar@{}[d]|{\eqref{left-dual}}\ar@{=}[dr] \ar[r]^{F(1 \ot 1\ot 1\ot s_2^{-1})} \ar@{}[ur]|{(N)} & F(S(Y \ot X) \ot Y \ot X \ot SX \ot SY) \ar[d]^{F(1\ot 1\ot 1\ot s_2)} 
\\
&& &F(S(Y \ot X) \ot Y \ot X \ot S(Y \ot X)) \ar[d]^{F(\e \ot 1)} 
\\
&& & FS(Y \ot X)  \\
&& & 
}}
\]


\subsection{Proof of Proposition~\ref{rem:leftaut=rightaut}}\label{proof2}
    
It is enough to check only one implication and one diagram, as the remaining will follow by duality. Assume that $F$ is left autonomous. Then the assertion follows from the commutative diagram~\eqref{proof:leftaut=rightaut} on page \pageref{proof:leftaut=rightaut}, using that both $\ka $ and $\beta$ are isomorphisms.


\subsection{Proof of Theorem~\ref{thm:aut->frob}}\label{proof3}

First, we show that the two monoidal structures induced on $F$ agree. This follows from the commutative diagram~\eqref{aut->frobMonoidal1=2} on page~\pageref{aut->frobMonoidal1=2}. Next, it is enough to check only one of the diagrams~\eqref{eq1:Frob},~\eqref{eq2:Frob} in the definition of Frobenius monoidal functor, the second one following by duality. This is done in the diagram~\eqref{last} on page~\pageref{last}, using that $\sigma:S'F^\op S \to F$ is a comonoidal-monoidal (iso)morphism (Remark~\ref{klst}). We shall abbreviate by $(S'FS)_2$ one of the structural morphism giving the comonoidal structure on $S'FS$, that is, the following composite:
\[
S'FS(X \ot Y) \overset{S'Fs_2}{\longrightarrow } S'F(SY \ot SX) \overset{S'f_2}{\longrightarrow } S'(FSY \ot FSX)\overset{{s'_2}^{-1}}{\longrightarrow} S'FSX \ot S'FSY
\]  


\begin{landscape}
\begin{equation}\label{largediagram}
\resizebox{23cm}{!}{\
\xymatrix@C=45pt{
& \ar[dl]_{Sf_2} SF(Y {\cdot} X) \ar[r]^{\ka} \ar[d]^{(1{\cdot} \db)(1{\cdot} 1{\cdot} \db {\cdot} 1)} & FS(Y {\cdot} X) \ar[d]_{(1{\cdot} \db)(1{\cdot} 1{\cdot} \db {\cdot} 1)} \ar[dr]^{1 {\cdot} f_0} 
\\
S(FY {\cdot} FX) \ar@{}[r]|{(M)} \ar[d]_{(1{\cdot} \db)(1{\cdot} 1{\cdot} \db {\cdot} 1)} & SF(Y {\cdot} X) {\cdot} FY {\cdot} FX {\cdot} SFX {\cdot} SFY 
\ar[dl]_{Sf_2 {\cdot} 1{\cdot} 1{\cdot} 1{\cdot} 1} \ar[d]^{1{\cdot} f_2{\cdot} 1{\cdot} 1}  & FS(Y {\cdot} X) {\cdot} FY {\cdot} FX {\cdot} SFX {\cdot} SFY \ar@{}[l]|{(M)}  \ar[d]^{1{\cdot} f_2{\cdot} 1{\cdot} 1} \ar[dr]^{1{\cdot} 1{\cdot} 1{\cdot} \ka{\cdot} \ka}  &  FS(Y {\cdot} X) {\cdot} F{\I} \ar[dr]^{1 {\cdot} F\db} 
\\
S(FY {\cdot} FX) {\cdot} FY {\cdot} FX {\cdot} SFX {\cdot} SFY \ar@{}[r]|{\eqref{eq2:dinat}} \ar[d]_{\e {\cdot} 1 {\cdot} 1} &SF(Y{\cdot} X) {\cdot} F(Y {\cdot} X) {\cdot} SFX {\cdot} SFY \ar@{}[dd]|{\eqref{eq1:ka-db}} \ar[r]^{\ka {\cdot} 1{\cdot} 1{\cdot} 1} \ar[dl]^{\e{\cdot} 1{\cdot} 1} & FS(Y{\cdot} X) {\cdot} F(Y {\cdot} X) {\cdot} SFX {\cdot} SFY \ar[d]^{f_2 {\cdot} 1{\cdot} 1} & FS(Y {\cdot} X) {\cdot} FY {\cdot} FX {\cdot} FSX {\cdot} FSY \ar[d]_{1{\cdot} f_2{\cdot} 1{\cdot} 1} & FS(Y {\cdot} X) {\cdot} F( Y {\cdot} X {\cdot} S(Y {\cdot} X)) \ar[ddd]^{1 {\cdot} F(1{\cdot} 1 {\cdot} s_2^{-1}) } \ar@{}[ddl]|{\eqref{ka-db-ot}}
\\
SFX {\cdot} SFY \ar[d]_{\ka {\cdot} \ka} \ar[rd]_{f_0 {\cdot} 1{\cdot} 1} & &F(S(Y {\cdot} X) {\cdot} Y {\cdot} X) {\cdot} SFX {\cdot} SFY  \ar@{}[dd]|{(M)} \ar[dl]^{F\e {\cdot} 1{\cdot} 1} &
FS(Y {\cdot} X) {\cdot} F(Y{\cdot} X) {\cdot} FSX {\cdot} FSY 
\ar[d]^{1{\cdot} 1{\cdot} f_2} & 
\\
FSX {\cdot} FSY \ar[d]_{f_2} &F{\I} {\cdot} SFX {\cdot} SFY \ar[d]^{1{\cdot} \ka {\cdot} \ka} & 
& FS(Y {\cdot} X) {\cdot} F(Y {\cdot} X) {\cdot} F(SX {\cdot} SY) \ar[dl]^{f_2 {\cdot} 1} 
 \ar[dr]_{1{\cdot} f_2} 
\\
F(SX {\cdot} SY) \ar@{}[r]|{(M)} \ar[d]_{Fs_2} &F{\I} {\cdot} FSX {\cdot} FSY \ar[d]^{1{\cdot} f_2}& F(S(Y {\cdot} X) {\cdot} Y {\cdot} X) {\cdot} F(SX {\cdot} SY) \ar[d]_{1 {\cdot} Fs_2} \ar[dr]^{f_2}  && FS(Y {\cdot} X) {\cdot} F(Y {\cdot} X {\cdot} SX {\cdot} SY) \ar[dl]_{f_2} 
\\
FS( Y{\cdot} X) \ar@/_2.5ex/@{=} [ddr] \ar[dr]^{f_0{\cdot} 1}  & F{\I} {\cdot} F(SX {\cdot} SY) \ar[d]^{1{\cdot} Fs_2} & F(S(Y{\cdot} X) {\cdot} Y {\cdot} X) {\cdot} FS( Y {\cdot} X)  \ar[d]_{f_2}  & F(S(Y {\cdot} X) {\cdot} Y {\cdot} X {\cdot} S(Y {\cdot} X)) \ar[dl]^{F(1{\cdot} 1{\cdot} 1 {\cdot} s_2)}  \ar@{}[l]|{(M)} \ar@{}[uu]|{\eqref{lax-functor2}}& & \\
&F{\I}{\cdot} FS(Y {\cdot} X) \ar@{}[l]|{\eqref{lax-functor0}} \ar[d]^{f_2} & F( S(Y {\cdot} X) {\cdot} Y {\cdot} X {\cdot} S(Y {\cdot} X))\ar[dl]_{F(\e {\cdot} 1)}  & &  &\\
&FS(Y {\cdot} X) & & & & \\
}}
\end{equation}
\end{landscape}


\begin{landscape}
\begin{equation}\label{proof:leftaut=rightaut}
\resizebox{22cm}{!}{
\xymatrix@C=40pt{
 & {\I} \ar[rrrr]^{f_0} \ar`l[dl]_{\db'F}[dl] \ar[d]^{\db S'F} \ar[dr]^{\db FS'} \ar@{}[dl]|{\eqref{d'-beta-d}} &  \ar@{}[drr]|{\eqref{eq1:lax_pres_dual}} && &  F{\I} \ar[dl]^{F\db S'} \ar`d[ddl]^{F\db'}[ddl] \\
S'FX \ot FX \ar[r]^{1 \ot \beta F} \ar`d[dr]_{\la \ot 1}[dr] & S'FX \ot SS'FX \ar[dr]^{\la \ot 1} \ar@{}[d]|{(M)} \ar@{}[r]|{\eqref{eq2:dinat}} & FS'X \ot SFS'X \ar@{}[dr]|{\eqref{eq:ka-la}}  \ar[d]_{1 \ot S \la} \ar[r]^{1 \ot \ka S'} &FS'X \ot FSS'X \ar[r]^{f_2} \ar[d]_{1 \ot F\beta^{-1} } \ar@{}[dr]|{(N)} & F(S'X \ot SS'X) \ar[d]^{F(1 \ot \beta^{-1})} \ar@{}[r]|-{\eqref{d'-beta-d}}& \\
& FS'X \ot FX \ar[r]^{1\ot \beta F} \ar@<.2ex>@{-}`d[rr]`[rr][rr] \ar@<-.5ex>@{-}`d[rr]`[rr][rr] & FS'X \ot SS'FX \ar[r]^{1 \ot \beta^{-1} F} & FS'X \ot FX  \ar[r]^{f_2} & F(S'X \ot X) & 
}}
\end{equation}
\end{landscape}


\begin{landscape}
\begin{equation}\label{aut->frobMonoidal1=2}
\resizebox{22cm}{!}{
\xymatrix@C=45pt{
&&&S'(FSY \p FSX) \ar@{}[ddr]|{\eqref{eq:lax-colax}} \ar[r]^{S'f_2} &S'F(SY \p SX) \ar[dr]^{S'Fs_2^{-1}}&&& \\
&&S'FSX \p S'FSY \ar@{}[d]|{(N)} \ar[ur]^{s'_2} &&&S'FS(X\p Y) \ar[dr]^{S'\ka^{-1} }&&\\
&S'SFX \p S'SFY \ar[ur]^{S'\ka \p S'\ka} \ar[r]^{s'_2} & S'(SFY \p SFX) \ar[uur]_{S'(\ka\p \ka)} &&&&S'SF(X\p Y) \ar[dr]^{\alpha^{-1} F} &\\
FX \p FY \ar@{}[urr]|{\eqref{eq:S2doctrinalS'2}} \ar@{=}[d] \ar[ur]^{\alpha F \p \alpha F}  \ar[r]^{\alpha} & S'S(FX \p FY) \ar@<1.5ex>@{}[drrrrr]|{(N)} \ar[ur]_-{S's_2} \ar[urrrrr]^{S'Sf_2} &&&&&&F(X\p Y)\ar@{=}[d] \\
FX\p FY \ar[rrrrrrr]^{f_2} \ar@{}[drrrrrrr]|{(N)} \ar@{=}[d]&&&&&&&F(X\p Y) \ar@{=}[d]\\
FX\p FY \ar[dr]_{F\beta\p F\beta} &&& &&&F(SS'X \p SS'Y) \ar[r]^{F(\beta^{-1} \p \beta^{-1})}  &F(X\p Y) \ar@{}[dll]|{\eqref{eq:S2doctrinalS'2}} \\
&FSS'X \p FSS'Y \ar[dr]_{\ka S'\p \ka S'} \ar[urrrrr]^{f_2} &&&&FS(S'Y \p S'X) \ar@{}[d]|{(N)} \ar[r]^{FS{s'_2}^{-1} } \ar[ur]^{Fs_2^{-1} }&FSS'(X\p Y) \ar[ur]_{F\beta^{-1}} &\\
&&SFS'X \p SFS'Y \ar[dr]_{s_2} &&&SFS'(X\p Y) \ar[ur]_{\ka^{-1} S'}&&\\
&&&S(FS'Y \p FS'X) \ar[r]^{Sf_2} &SF(S'Y \p S'X) \ar[ur]_{SF{s'_2}^{-1} } \ar[uur]^{\ka^{-1} S'} \ar@{}[uul]|{\eqref{eq:lax-colax}}&&&
}
}
\end{equation}
\end{landscape}

\begin{landscape}
\begin{equation}
\resizebox{22cm}{!}{
\label{last}
\xymatrix{& F(X\p Y)\p FZ \ar`l[dddd]`[dddd]`[rr]_{F_2\p 1} [dddrr] \ar[r]^{f_2} \ar[drr]_-{\sigma^{-1} \p \sigma^{-1}} 
\ar[d]_{\sigma^{-1} \p 1} & F(X \p Y\p Z) \ar`u`[rrrrr]^{F_2}`[ddd][dddrrrr] \ar@<1.2ex>@{}[dr]|{\eqref{lax-colax}} \ar[r]^{\sigma^{-1} } & 
S'FS(X \p Y \p Z) \ar[d]^{(S'FS)_2} \ar[r]^{(S'FS)_2} & S'FSX \p S'FS(Y \p Z) \ar[d]^{1 \p (S'FS)_2} \ar`r[dddrr][dddrr]^{\sigma \p \sigma} \ar@{}[dl]|{\eqref{colax_functor2}} && & \\
& S'FS(X \p Y ) \p FZ \ar[dd]_{(S'FS)_2\p 1} \ar[rr]_{1 \p \sigma^{-1}} & & S'FS(X \p Y) \p S'FSZ \ar[r]^{(S'FS)_2 \p 1} & S'FSX  \p S'FSY \p S'FSZ \ar[dr]^{\sigma\p \sigma\p \sigma} && \\
& &&&& FX \p FY \p FZ \ar@{=}[dll] \ar@{}[uu]|{(M),\eqref{lax-colax}} \ar[dr]^{1 \p f_2} \\
& S'FSX \p S'FSY \p FZ \ar[uurrr]^{1\p 1\p \sigma^{-1}} \ar@<.5ex>@{}[uurr]|{(M)} \ar[rr]^-{\sigma \p \sigma \p 1} && FX \p FY \p FZ \ar[rrr]^{1 \p f_2} && &FX \p F(Y \p Z)  \\
& &&&&&}
}
\end{equation}
\end{landscape}

\end{appendices}


\bibliographystyle{alpha}
\bibliography{BibFile1}


\end{document}